\documentclass{amsart}

\newtheorem{thm}{Theorem}[section]

\newtheorem{cor}[thm]{Corollary}
\newtheorem{lem}[thm]{Lemma}

\theoremstyle{definition}

\newtheorem{rem}[thm]{Remark}       %%%%% the counter [thm] is optional
\usepackage{amsmath, amssymb, mathrsfs, verbatim, multirow}
\newcommand{\mathsym}[1]{{}}

\numberwithin{equation}{section}

\begin{document}

\title[multivariate moment problem]{Application of localization to the multivariate moment problem II}

\author{Murray Marshall}

\address{Department of Mathematics and Statistics,
University of Saskatchewan, %\newline \indent
Saskatoon, SK S7N5E6, Canada}
\email{marshall@math.usask.ca}
\thanks {This research was funded in part by an NSERC of Canada Discovery Grant. %, and by the Faculty of Mathematics and Physics, University of Ljubljana.
}

\begin{abstract} The paper is a sequel to the paper \cite{M3} by the same author. A new criterion is presented for a PSD linear map $L : \mathbb{R}[\underline{x}] \rightarrow \mathbb{R}$ to correspond to a positive Borel measure on $\mathbb{R}^n$. The criterion is stronger than Nussbaum's criterion in \cite{N} and is similar in nature to %seems to be incomparable to
Schm\"udgen's criterion in \cite{M3} \cite{S}.
It is also explained how the criterion allows one to understand the support of the associated measure in terms of the non-negativity of $L$ on a quadratic module of $\mathbb{R}[\underline{x}]$. This latter result extends a result of Lasserre in \cite{L}. The techniques employed are the same localization techniques employed already in \cite{M1} and  \cite{M3}, specifically one works in the localization of $\mathbb{R}[\underline{x}]$ at $p = \prod_{i=1}^n(1+x_i^2)$ or $p' = \prod_{i=1}^{n-1}(1+x_i^2)$.
\end{abstract}

\maketitle

%\section{Introduction}

This paper is a sequel to the earlier paper \cite{M3}. We present a couple of interesting and illuminating results which were inadvertently overlooked when \cite{M3} was written; see Theorems \ref{new} and \ref{support} below. Theorem \ref{new} extends an old result of Nussbaum in \cite{N}. See Theorem \ref{md carleman} below for a statement of Nussbaum's result. The density condition ({\ref{n}) appearing in Theorem \ref{new} is weaker than the Carleman condition (\ref{v})
appearing in Nussbaum's result. Theorem \ref{support} shows how condition (\ref{n}) allows one to read off information about the support of the measure from the non-negativity of the linear functional on a quadratic module. This illustrates how natural condition (\ref{n}) is.  Theorem \ref{support} extends a result of Lasserre in \cite{L}.

We recall some terminology and notation from \cite{M1} and \cite{M3}. For an $\mathbb{R}$-algebra $A$ (commutative with 1), a \it quadratic module \rm of $A$ is a subset $M$ of $A$ such that $1\in M$, $M+M\subseteq M$ and $f^2M\subseteq M$ for all $f\in A$. $\sum A^2$ denotes the set of all (finite) sums of squares of $A$. $\sum A^2$ is the unique smallest quadratic module of $A$. A linear map $L : A \rightarrow \mathbb{R}$ is said to be PSD (positive semidefinite) if $L(f^2) \ge 0$ for all $f\in A$, equivalently, if $L(\sum A^2) \subseteq [0,\infty)$. Define $\mathbb{R}[\underline{x}] := \mathbb{R}[x_1,\dots,x_n]$, $\mathbb{C}[\underline{x}] := \mathbb{C}[x_1,\dots,x_n]$.  If $\mu$ is a positive Borel measure on $\mathbb{R}^n$ having finite moments, i.e., $\int \underline{x}^k d\mu$ is well-defined and finite for all monomials $\underline{x}^k := x_1^{k_1}\dots x_n^{k_n}$, $k_j\ge 0$, $j=1,\dots,n$, the PSD linear map $L_{\mu} : \mathbb{R}[\underline{x}] \rightarrow \mathbb{R}$ is defined by $L_{\mu}(f) = \int f d\mu$. If $\nu$ is another positive Borel measure on $\mathbb{R}^n$ having finite moments then we write $\mu \sim \nu$ is indicate that $\mu$ and $\nu$ have the same moments, i.e., $L_{\mu}=L_{\nu}$. We say $\mu$ is \it determinate \rm if $\mu \sim \nu$ $\Rightarrow$ $\mu = \nu$.

\begin{thm} \label{new} Suppose $L : \mathbb{R}[\underline{x}] \rightarrow \mathbb{R}$ is linear and PSD and, for $j=1,\dots,n-1$,
\begin{equation}\label{n}
\exists \text{ a sequence } \{q_{jk}\}_{k=1}^{\infty} \text{ in } \mathbb{C}[\underline{x}] \text{ such that } \lim\limits_{k\rightarrow\infty}L(|1-(1+x_j^2)q_{jk}\overline{q_{jk}}|^2)=0.
 \end{equation}
Then there exists a positive Borel measure $\mu$ on $\mathbb{R}^n$ such that $L=L_{\mu}$. If condition (\ref{n}) holds also for $j=n$ then the measure is determinate.
\end{thm}

\begin{proof} Extend $L$ to $\mathbb{C}[\underline{x}]$ in the obvious way, i.e., $L(f_1+if_2) := L(f_1)+iL(f_2)$. Define $\langle f,g\rangle := L(f\overline{g})$, $\| f\| := \sqrt{\langle f,f\rangle}$. According to \cite[Corollary 4.8]{M3} to prove the existence assertion it suffices to show that $\forall$ $g\in \mathbb{C}[\underline{x}]$ and $\forall$ $j= 1,\dots,n-1$, $$\lim\limits_{k\rightarrow \infty}L(g(1-(1+x_j^2)q_{jk}\overline{q_{jk}}))= 0.$$ This is immediate from condition (\ref{n}), using the Cauchy-Schwartz inequality.
According to \cite[Corollary 2.7]{M3}, to show uniqueness it suffices to show $\forall$ $j = 1,\dots,n$
$\exists$ a sequence $\{p_{jk}\}_{k=1}^{\infty}$ in $\mathbb{C}[\underline{x}]$ such that  $$\lim\limits_{k\rightarrow\infty}L(|1-(x_j-i)p_{jk}|^2)=0.$$
Uniqueness follows from this criterion, taking $p_{jk} := (x_j+i)q_{jk}\overline{q_{jk}}$.
\end{proof}

We remark that \cite[Theorem 4.9]{M3} is a consequence of Theorem \ref{new}. This is immediate from the following:

\begin{lem}\label{clarify} Suppose $L : \mathbb{R}[\underline{x}] \rightarrow \mathbb{R}$ is linear and PSD. Suppose $\{q_{jk}\}_{k=1}^{\infty}$ is a sequence of polynomials in $\mathbb{C}[\underline{x}]$. Then $$\lim\limits_{k\rightarrow \infty} L(| 1-(x_j-i)q_{jk}|^4) = 0 \ \Rightarrow \ \lim\limits_{k\rightarrow \infty}L(|1-(1+x_j^2)q_{jk}\overline{q_{jk}}|^2) =0.$$
\end{lem}

\begin{proof}  Let $Q_k := 1-(x_j-i)q_{jk}$. Thus $$1-(1+x_j^2)q_{jk}\overline{q_{jk}} = 1-(1-Q_k)(1-\overline{Q}_k) = Q_k+\overline{Q}_k-Q_k\overline{Q_k}.$$ We are assuming $\|Q_k\overline{Q_k}\| \rightarrow 0$ as $k \rightarrow \infty$ and we want to show $\| Q_k+\overline{Q_k}-Q_k\overline{Q_k}\| \rightarrow 0$ as $k \rightarrow \infty$. Applying the Cauchy-Schwartz inequality and the triangle inequality we obtain $\| Q_k\|^2 = \| \overline{Q_k}\|^2 = \langle Q_k\overline{Q_k},1\rangle \le \|Q_k\overline{Q_k}\|\cdot \| 1 \|$ and $$\| Q_k+\overline{Q_k}-Q_k\overline{Q_k}\| \le \| Q_k\| +\| \overline{Q_k}\|+\| Q_k\overline{Q_k}\| \le 2\sqrt{\|Q_k\overline{Q_k}\|\cdot \| 1 \|}+\| Q_k\overline{Q_k}\|.$$ At this point the result is clear.
\end{proof}

The following result of Nussbaum \cite[Theorem 4.11]{N} can also be %viewed
seen as a consequence of Theorem \ref{new}.

\begin{thm}[Nussbaum] \label{md carleman}  Suppose $L : \mathbb{R}[\underline{x}] \rightarrow \mathbb{R}$ is linear and PSD and, for $j=1,\dots,n-1$, the Carleman condition
\begin{equation}\label{v}
\sum_{k=1}^{\infty} \frac{1}{\root 2k \of{L(x_j^{2k})}} = \infty
\end{equation}
holds. Then there exists a positive Borel measure $\mu$ on $\mathbb{R}^n$ such that $L=L_{\mu}$. If condition (\ref{v}) holds also for $j=n$ then the measure is determinate.
\end{thm}

\begin{proof} Argue as in \cite[Theorem 4.10]{M3}. Let $\mu_j$ be  the positive Borel measure on $\mathbb{R}$ such that $L_{\mu_j} = L|_{\mathbb{R}[x_j]}$. According to \cite[Th\'eor\`eme 3]{BC}, the Carleman condition (\ref{v}) implies that $\mathbb{C}[x_j]$ is dense in the Lebesgue space $\mathcal{L}^s(\mu_j)$ for all $s\in [1,\infty)$. Fix $s>4$. Thus $\exists$ $q_{jk} \in \mathbb{C}[x_j]$ such that $\lim\limits_{k\rightarrow\infty}\| q_{jk}-\frac{1}{x_j-i}\|_{s,\mu_i}=0$.  An easy application of H\"older's inequality (taking $p= \frac{s}{4}$, $q= \frac{s}{s-4}$) shows that
\begin{align*}
L(|1-(x_j-i)q_{jk}|^4) &= \int |q_{jk}-\frac{1}{x_j-i}|^4|x_j-i|^4d\mu_j \\ &\le \big[ \| q_{jk}-\frac{1}{x_j-i}\|_{s,\mu_j} \cdot \| x-i\|_{\frac{4s}{s-4}, \mu_j}\big]^4
\end{align*}
so $\lim\limits_{k\rightarrow\infty}L(|1-(x_j-i)q_{jk}|^4) =0$.  The result follows now, by Lemma \ref{clarify} and Theorem \ref{new}.
\end{proof}

The reader should compare Theorems \ref{new} and \ref{md carleman} with the following result of Schm\"udgen \cite[Theorem 4.11]{M3} \cite[Proposition 1]{S}, which, according to Fuglede \cite[p. 62]{F}, is an unpublished result of J.P.R. Christensen, 1981. %:\footnotemark\footnotetext{According to Fuglede \cite[p. 62]{F}, Theorem \ref{strong} is also an unpublished result of J.P.R. Christensen, 1981.}

\begin{thm}[Schm\"udgen] \label{strong} Suppose $L : \mathbb{R}[\underline{x}] \rightarrow \mathbb{R}$ is linear and PSD. Fix a positive Borel measure $\mu_j$ on $\mathbb{R}$ such that $L|_{\mathbb{R}[x_j]} = L_{\mu_j}$ and suppose for $j = 1,\dots,n-1$ that $\mathbb{C}[x_j]$ is dense in $\mathcal{L}^4(\mu_j)$, i.e.,
\begin{equation}\label{s} \exists \text{ a sequence } \{q_{jk}\}_{k=1}^{\infty} \text{ in } \mathbb{C}[x_j] \text{ such that } \lim\limits_{k\rightarrow\infty}\|q_{jk}-\frac{1}{x_j-i}\|_{4,\mu_j}=0.
\end{equation}
Then there exists a positive Borel measure $\mu$ on $\mathbb{R}^n$ such that $L=L_{\mu}$. If condition (\ref{s}) holds also for $j=n$ then the measure is determinate.
\end{thm}

By considering products of measures of the sort considered by Sodin in \cite{So}, one sees that Theorem \ref{new} and Theorem \ref{strong} are both strictly stronger than Nussbaum's result. But it is not clear, to the author at least, how Theorems \ref{new} and %\cite[Theorem 4.11]{M3}
\ref{strong} are related. In particular, it is not clear that either result implies the other.%\footnotemark\footnotetext{Actually, as pointed out already by Schm\"udgen in \cite[Open Problem a]{S}, there is still a big gap here, between what one is able prove and the known counterexamples.}

We turn now to the problem of describing the support of $\mu$. By definition, the support of $\mu$ is the smallest closed set $K$ of $\mathbb{R}^n$ satisfying $\mu(\mathbb{R}^n \backslash K) = 0$. We recall additional notation from \cite{M1} and \cite{M3}. If $M$ is a quadratic module of an $\mathbb{R}$-algebra $A$, define $$X_M := \{ \alpha : A \rightarrow \mathbb{R} \mid \alpha \text{ is an } \mathbb{R}\text{-algebra homomorphism, } \alpha(M) \subseteq [0,\infty)\}.$$ If $M = \sum A^2+I$, where $I$ is an ideal of $A$, the condition $\alpha(M) \subseteq [0,\infty)$ is equivalent to the condition $\alpha(I) = \{ 0 \}$. $\mathbb{R}[\underline{x}]_p$ denotes the localization of $\mathbb{R}[\underline{x}]$ at $p$ where $p := \prod_{j=1}^n (1+x_j^2)$. If $A$ is $\mathbb{R}[\underline{x}]$ or $\mathbb{R}[\underline{x}]_p$ then algebra homomorphisms $\alpha : A \rightarrow \mathbb{R}$ are identified with points of $\mathbb{R}^n$ via the map $\alpha \mapsto (\alpha(x_1),\dots,\alpha(x_n))$ and  $X_M$ is identified  with the set $\{ \underline{a} \in \mathbb{R}^n \mid g(\underline{a})\ge 0 \ \forall \ g\in M\}$.

\begin{thm} \label{support} Suppose $L : \mathbb{R}[\underline{x}] \rightarrow \mathbb{R}$ is a PSD linear map satisfying (\ref{n}) for $j=1,\dots,n$ and $g\in \mathbb{R}[\underline{x}]$ is such that $L(gh^2)\ge 0$ $\forall$ $h\in \mathbb{R}[\underline{x}]$. Then the support of the associated positive Borel measure $\mu$ is contained in the set $\{ \underline{a} \in \mathbb{R}^n \mid g(\underline{a})\ge 0\}$.
\end{thm}

See \cite[Theorem 2.2]{L} for an earlier version of this result.

\begin{proof} Denote by $L : \mathbb{R}[\underline{x}]_p \rightarrow \mathbb{R}$ the PSD linear extension of $L$ defined by $L(f) := \int f d\mu$ $\forall$ $f\in \mathbb{R}[\underline{x}]_p$. We claim that $L(gh\overline{h})\ge 0$ $\forall$ $h \in \mathbb{C}[\underline{x}]_p$ (so, in particular, $L(gh^2)\ge 0$ $\forall$ $h\in \mathbb{R}[\underline{x}]_p$). The proof is by induction of the number of factors of the form $x_j\pm i$, $j=1,\dots,n$ appearing in the denominator of $h$. Suppose $x_j\pm i$ appears in the denominator of $h$. Note that $(x_j \pm i)hq_{jk}$ has fewer factors $x_j\pm i$ appearing in the denominator, so, by induction, $L(g(1+x_j^2)h\overline{h}q_{jk}\overline{q_{jk}})\ge 0$. Applying the Cauchy-Schwartz inequality, we see that $L(gh\overline{h}(1-(1+x_j^2)q_{jk}\overline{q_{jk}})) \rightarrow 0$ as $k\rightarrow \infty$. It follows that $L(g(1+x_j^2)h\overline{h}q_{jk}\overline{q_{jk}}) \rightarrow L(gh\overline{h})$ as $k \rightarrow \infty$, so $L(gh\overline{h})\ge 0$. This proves the claim. Denote by $Q$ the quadratic module of $\mathbb{R}[\underline{x}]_p$ generated by $g$, i.e., $Q := \sum \mathbb{R}[\underline{x}]_p^2+\sum \mathbb{R}[\underline{x}]_p^2g$. It follows from the claim together with the fact that $L$ is PSD on $\mathbb{R}[\underline{x}]_p$ that $L(Q)\subseteq [0,\infty)$. By \cite[Corollary 3.4]{M1} there exists a positive Borel measure $\nu$ on $X_Q = \{ \underline{a} \in \mathbb{R}^n \mid g(\underline{a})\ge 0\}$ such that $L(f) = \int f d\nu$ $\forall$ $f \in \mathbb{R}[\underline{x}]_p$. Uniqueness of $\mu$ implies $\mu = \nu$.
\end{proof}

\begin{cor}\label{cor} If $L$ satisfies condition (\ref{n}) for $j=1,\dots,n$ and $L(M)\subseteq [0,\infty)$ for some quadratic module $M$ of $\mathbb{R}[\underline{x}]$ then the support of the associated positive Borel measure $\mu$ is contained in the set $X_M = \{ \underline{a} \in \mathbb{R}^n \mid g(\underline{a})\ge 0 \ \forall \ g \in M\}$.
\end{cor}

\begin{rem} \

(1) The quadratic module $M$ is not required to  be finitely generated, although this seems to be the most interesting case.

(2) For a quadratic module of the form $M = \sum \mathbb{R}[\underline{x}]^2+I$, $I$ an ideal of $\mathbb{R}[\underline{x}]$, one can weaken the hypothesis. It is no longer necessary to assume that $L$ satisfies condition (\ref{n}) for $j=1,\dots,n$ but only that $L=L_{\mu}$. This is more or less clear. By the Cauchy-Schwartz inequality, for $g\in \mathbb{R}[\underline{x}]$, $$L(gh)=0 \ \forall \ h\in \mathbb{R}[\underline{x}] \ \Leftrightarrow \ L(g^2)=0 \ \Leftrightarrow \ L(gh)=0 \ \forall \ h\in \mathbb{R}[\underline{x}]_p.$$ Also, in this case, $X_M= Z(I)= \{ \underline{a} \in \mathbb{R}^n \mid g(\underline{a})=0 \ \forall \ g \in I\}$.
\end{rem}


\begin{thebibliography}{99}
\bibitem{BC} C. Berg, J.P.R. Christensen, Exposants critiques dans le probl\'eme des moments. \textit{C.R. Acad. Sc. Paris S\'erie I} {\bf 296}, 661--663 (1983).
\bibitem{F} B. Fuglede, The multidimensional moment problem. \textit{Expositiones Mathematicae} {\bf1}, 47--65 (1983).
\bibitem{L} J.B. Lasserre, The K-moment problem for continuous functionals. \textit{Trans. Amer. Math. Soc.} {\bf 365}, 2489--2504 (2013).
\bibitem{M1} M. Marshall, Approximating positive polynomials using sums of squares, \textit{Canad. Math. Bull.} {\bf 46}, 400--418 (2003).
\bibitem{M3} M. Marshall, Application of localization to the multivariate moment problem. \textit{Math. Scand.}, to appear.
\bibitem{N} A.E. Nussbaum, Quasi-analytic vectors, \textit{Ark. Math.} {\bf 6} (10), 179--191 (1965).
\bibitem{S} K. Schm\"udgen, On determinacy notions for the two dimensional moment problem, \textit{Ark. Math.} {\bf 29}, 277--284 (1991).
\bibitem{So} M. Sodin, A note on the Hall-Mergelyan theme. \textit{Mat. Fiz. Anal. Geom.} {\bf 3} no. 1-2, 164--168 (1996).
\end{thebibliography}
\end{document}